\documentclass[draft, reqno]{amsart}
\usepackage[all]{xy}                        %

\CompileMatrices                            

\UseTips                                    

\input xypic
\usepackage[bookmarks=true]{hyperref}       

\usepackage{amssymb,latexsym,amsmath,amscd}
\usepackage{xspace}
\usepackage{enumerate}
\usepackage{graphicx}
\usepackage{vmargin}
\usepackage{todonotes}

\reversemarginpar

\theoremstyle{plain}
\newtheorem{theorem}{Theorem}[section]
\newtheorem*{theorem*}{Theorem}
\newtheorem{proposition}[theorem]{Proposition}
\newtheorem{corollary}[theorem]{Corollary}
\newtheorem{lemma}[theorem]{Lemma}

\theoremstyle{definition}
\newtheorem{definition}[theorem]{Definition}

\newtheorem{remark}[theorem]{Remark}
\newtheorem{example}[theorem]{Example}

\newcommand{\enm}[1]{\ensuremath{#1}}          %

\newcommand{\cal}[1]{\mathcal{#1}}

\renewcommand{\bar}[1]{\overline{#1}}

\newcommand{\CC}{\enm{\mathbb{C}}}

\newcommand{\RR}{\enm{\mathbb{R}}}

\newcommand{\PP}{\enm{\mathbb{P}}}

\newcommand{\HH}{\enm{\mathbb{H}}}

\newcommand{\Dd}{\enm{\cal{D}}}

\newcommand{\Ff}{\enm{\cal{F}}}

\newcommand{\Jj}{\enm{\cal{J}}}

\newcommand{\Oo}{\enm{\cal{O}}}

\newcommand{\Qq}{\enm{\cal{Q}}}

\newcommand{\Ss}{\enm{\cal{S}}}

\newcommand{\Uu}{\enm{\cal{U}}}

\renewcommand{\phi}{\varphi}
\renewcommand{\theta}{\vartheta}
\renewcommand{\epsilon}{\varepsilon}

\begin{document}
\title[Three Results on the Tw. Disc. Locus in the 4-Sphere]{Three Topological Results on the Twistor Discriminant Locus in the 4-Sphere}
\author[A. Altavilla]{A. Altavilla${}^{\ddagger}$}\address{Altavilla Amedeo: Dipartimento Di Matematica, Universit\`a di Roma ``Tor Vergata", Via Della Ricerca Scientifica 1, 00133, Roma, Italy} \email{altavilla@mat.uniroma2.it}

\author[E. Ballico]{E. Ballico${}^{\dagger}$}\address{} \email{}\address{Edoardo Ballico: Dipartimento Di Matematica, Universit\`a di Trento, Via Sommarive 14, 38123, Povo, Trento, Italy} \email{edoardo.ballico@unitn.it}

\thanks{${}^{\dagger,\ddagger}$GNSAGA of INdAM;
 ${}^{\dagger}$MIUR PRIN 2015 ``Geometria delle variet\`a algebriche'';
 ${}^{\ddagger}$SIR grant {\sl ``NEWHOLITE - New methods in holomorphic iteration''} n. RBSI14CFME and SIR grant {\sl AnHyC - Analytic aspects in complex and hypercomplex geometry} n. RBSI14DYEB}

\date{\today }


\begin{abstract} 
We exploit techniques from classical (real and complex) algebraic geometry for the study of the 
standard twistor fibration $\pi:\mathbb{CP}^{3}\to S^{4}$. We prove three results about the topology of
the twistor discriminant locus of an algebraic surface in $\mathbb{CP}^{3}$. First of all we prove that,
with the exception of two exceptional cases, the real dimension of the twistor discriminant locus of an algebraic surface
is always equal to 2. Secondly we describe the possible intersections of a general surface with the family of twistor lines: 
we find that only 4 configurations are possible and for each of them we compute the dimension. Lastly we give
a decomposition of the twistor discriminant locus of a given cone in terms of its singular locus and its dual variety.
\end{abstract}
\maketitle


\section{Introduction and Algebraic OCS's}

Given an oriented Riemannian manifold $(M,g)$, its twistor space $Z(M)$ is the fibre bundle parametrizing orthogonal 
complex structures (OCS's) defined on $M$, i.e.: integrable complex structures, compatible 
with the metric $g$ and the orientation of $M$. In the case in which $M$ has dimension 4, it is possible to define an almost 
complex structure $\Jj$ on $Z(M)$ which turns out to be integrable if and only if $M$ is \textit{anti-self-dual}, 
i.e. the self-dual part $W_{+}$ of its Weyl tensor vanishes~\cite{ahs}. Moreover a complex 3-manifold $Z$ is the twistor space of some
Riemannian manifold if and only if it admits a fixed-point-free anti-holomorphic involution $j:Z\to Z$ and a foliation
by $j$-invariant rational curves isomorphic to $\CC\PP^{1}$, each of which has normal bundle $\Oo(1)\oplus\Oo(1)$~\cite{ahs,lebrun}.

Clearly, if $J$ is an OCS on $(M,g)$, then it is an OCS also on $(M, \tilde g)$, for any $\tilde g$ in the same conformal class of $g$. Therefore, the theory of twistor spaces is invariant under conformal transformations of the base space.

In some case the twistor space turns out to be an algebraic manifold. In dimension 4 this happens only for the 4-sphere
$S^{4}$ and the complex projective plane $\CC\PP^{2}$ together with their standard metrics. In this paper we will only focus on the case of
$S^{4}$ endowed with the standard round metric that will be now described in some detail. Let $\HH\PP^{1}$ denotes the left quaternionic projective line.
The twistor space of $S^{4}\simeq \RR^{4}\cup\{\infty\}\simeq \HH\PP^{1}$ is $\CC\PP^{3}$ together with its standard complex structure. The twistor fibration $\pi:\CC\PP^{3}\to\HH\PP^{1}\simeq S^{4}$ is given by
$$
\pi[z_{0},z_{1},z_{2},z_{3}]=[z_{0}+z_{1}j, z_{2}+z_{3}j],
$$
where $j$ is the standard quaternionic imaginary unit orthogonal to $i\in\HH$ and such that $ij=k$.
The $j:\CC\PP^{3}\to\CC\PP^{3}$ fixed-point-free anti-holomorphic involution is defined to be the function on $\CC\PP^{3}$ induced by the left multiplication by $j$ in $\HH\PP^{1}$ (and for this reason is denoted with the same symbol), i.e.:
$$
j[z_{0},z_{1},z_{2},z_{3}]:=[-\overline{z_{1}},\overline{z_{0}},-\overline{z_{3}},\overline	{z_{2}}].
$$
The fibers of $\pi$, also called \textit{twistor lines}, are then the projective lines fixed by $j$. Explicitly, if $q_{1},q_{2}\in\CC$, the fiber over $[1,q_{1}+q_{2}j]$ is given by the following system of equations in $\CC\PP^{3}$:
$$
\begin{cases}
z_{2}=z_{0}q_{1}-z_{1}\bar q_{2},\\
z_{3}=z_{0}q_{2}+z_{1}\bar q_{1}.
\end{cases}
$$

Any OCS on a domain $\Omega\subseteq \mathbb{H}$, can be represented as a $\CC\PP^{1}$-valued function or, since $\CC\PP^{1}\simeq SO(4)/U(2)$, as 
a matrix $J\in SO(4)$, such that $J=-^{t}J$ (see e.g.~\cite[Chapter 1]{debartolomeisnannicini})
\begin{equation*}
J=\left(\begin{matrix}
0 & A & B & C\\
-A & 0 & C & -B\\
-B & -C & 0 & A\\
-C & B & A & 0
\end{matrix}
\right).
\end{equation*}

Following the construction in~\cite[Section 2]{sv1} (see also~\cite[Section 2]{altavillatwistor}), in the affine set $\{z_{0}\neq 0\}\subset\CC\PP^{3}$, the OCS associated to the point $[1,z_{1}=x+iy,z_{2},z_{3}]$, is given by the following $4\times 4 $
real matrix
\begin{equation}\label{matrixJ}
J=\frac{-1}{1+|z_{1}|^{2}}\left(\begin{matrix}
0 & 1- |z_{1}|^{2} & 2y & -2x\\
-1+|z_{1}|^{2} & 0 & -2x & -2y\\
-2y & 2x & 0 & 1-|z_{1}|^{2}\\
2x & 2y & 1- |z_{1}|^{2} & 0
\end{matrix}
\right).
\end{equation}
Conformal transformations of $S^{4}$ are exactly M\"obius transformations of $\HH\PP^{1}$. It is possible to lift such transformations on $\CC\PP^{3}$ via $\pi$ and identify them as the set of
complex projective transformations commuting with $j$ (see~\cite[Section 2]{armstrong} for an explicit description).

The main starting result of our research is the following which shows just a glimpse of the deep relation given by the twistor fibration.
\begin{theorem}[\cite{ahs, debartolomeisnannicini}]\label{OCS}
Let $\Omega\subset\RR^{4}$ be a domain. 
\begin{enumerate}
\item If $J$ is an OCS on $\Omega$ then the graph in the twistor space $J(\Omega)\subset \CC\PP^{1}\times \CC^{2}$ is a holomorphic submanifold.
\item Let $\Ss\subset \pi^{-1}(\Omega)$ be a holomorphic submanifold such that for all $q\in\Omega$ $|\Ss\cap\pi^{-1}(q)|=1$, then $\Ss$ is the graph of an OCS.
\end{enumerate}
\end{theorem}

After this general theorem, a number of results were given on the conformal classification of surfaces in $\CC\PP^{3}$ and on couples $(J,\Omega)$,
where $J$ is an OCS defined on a maximal domain $\Omega\subset S^{4}$ \cite{altavillatwistor, ab, altavillasarfatti, armstrong, APS, sv2, chirka, gensalsto, sv1, shapiro}. In all the cited paper the authors deal, mostly, with OCS's
arising from algebraic surfaces in $\CC\PP^{3}$ and analyze them looking at the so-called \textit{twistor discriminant locus}.

\begin{definition}
Let $X\subset\CC\PP^{3}$ be any algebraic hypersurface of degree $d$. The \textit{twistor discriminant locus} of $X$ is defined
as the following subset of $S^{4}$:
$$
\mathrm{Disc}(X):=\{q\in S^{4}\,|\, |\pi^{-1}(q)\cap X|\neq d\}.
$$
\end{definition}

Since the fibers of $\pi$ are lines and $X$ has degree $d$, then the general intersection between $X$ and a fiber is composed
by $d$ different points. Hence, if $\pi$ is restricted to $X\setminus \pi^{-1}(\mathrm{Disc}(X))$, then it defines a degree $d$ unramified covering over $S^{4}\setminus \mathrm{Disc}(X)$. The twistor discriminant locus has many important useful
properties: broadly speaking, its topology is invariant under conformal transformations. For instance, each twistor line $L=\pi^{-1}(q)$ contained in a given surface $X$, produces a point $q\in \mathrm{Disc}(X)$ and the number of twistor lines contained in $X$ is 
a conformal invariant of the surface itself.

\begin{remark}\label{sing}
Clearly, given any algebraic surface $X$, its twistor discriminant locus can be seen as the projection on $S^{4}$ of the twistor lines
tangent to $X$ (or even contained in $X$) and of the singular locus $\mathrm{Sing}(X)$. 
\end{remark}

\begin{remark}
Given any degree $d$ algebraic surface $X$, its twistor discriminant locus $\mathrm{Disc}(X)$ is a nonempty real algebraic submanifold of $S^{4}$. See~\cite[Section 3, Remark (i)]{armstrong} for this result and for an estimate of its degree.
\end{remark}

For topological reasons there is not a global section for $\pi$ (there is not a global OCS on $S^{4}$). Therefore,
the simplest case is described in the following example.

\begin{example}\label{hyperplane}
If $X$ is an hyperplane, i.e. $d=1$, then $X$ contains a unique twistor line $L=\pi^{-1}(q)$. It is possible to prove that all hyperplanes are conformally equivalent and induce conformally constant OCS's on $S^{4}\setminus\{q\}$ \cite{shapiro}.
\end{example}

\begin{example}
The Segre quadric $\Qq:=\{z_{0}z_{3}=z_{1}z_{2}\}$ is such that $j(\Qq)=\Qq$. Moreover $\mathrm{Disc}(\Qq)=S^{1}\subset S^{4}$. In~\cite{sv1} it is proved that the identification of a (round) $\tilde S^{1}$ inside the 
4-sphere is sufficient to define the unique non-singular $j$-invariant  quadric $\tilde \Qq$, such that $\mathrm{Disc}(\tilde \Qq)=\tilde S^{1}$ and for each $q\in\tilde S^{1}$, $\pi^{-1}(q)\subset \tilde \Qq$. These twistor lines contained in $\tilde \Qq$ form one of its two rulings.
\end{example}

 For the reasons above we started to study abstract properties of algebraic surfaces
under the perspective of the twistor projection. These particular holomorphic surfaces of $\CC\PP^{3}$ induce what we call, in the following general definition, \textit{algebraic OCS's}.

\begin{definition}\label{aocs}
Let $(M, [g])$ be a conformal hermitian manifold and $\Omega\subseteq M$ be an open subdomain. 
An OCS $J$ on $\Omega$ such that its graph in the twistor space $J(\Omega)$
is contained in an algebraic manifold, is said to be an \textit{algebraic OCS}.
%
\end{definition}

Thanks to Liouville's Theorem, the definition of algebraic OCS is well posed.
In fact, any conformal map on an open set of the base space is necessarily defined on $M$ and therefore lift
to the automorphism group of $Z(M)$. If $Z(M)$ is not an algebraic manifold it is possible to find algebraic OCS's
as in the case in which $M$ is a scalar-flat K\"ahler  
surface. In this case the twistor space contains an algebraic hypersurface biholomorphic to $M$ itself (see~\cite{klp,lebrun,lebrun1,lebrun2,pontecorvo}).

Moreover, in~\cite[Section 16, arXiv version v1]{sv1} it is given an explicit example of an OCS
which is not algebraic. 
%

\begin{remark}
Clearly, if $\Ss\subset\pi^{-1}(\Omega)$ is (contained in) an algebraic submanifold such that, for all $q\in\Omega$, $|\Ss\cap\pi^{-1}(q)|=1$, then $\Ss$ is the graph of an algebraic OCS.
\end{remark}

In the cases in which the twistor space of a given manifold is algebraic, we expect to identify algebraic OCS's
as those which have real algebraic coordinates. A simple situation for $S^{4}$ is described in what follows.
Let $X\subset \CC\PP^{3}$ be an algebraic surface, such that $X$ intersects each twistor line in 
exactly one point. Then the induced OCS (in Equation~\ref{matrixJ}), has real algebraic coordinates and $\Omega$ is a real Zariski open set.
Identifying $\RR^{4}$ with $\CC^{2}$ and relaxing a bit the hypotheses on $\Omega$ it is possible to obtain the following sufficient condition.

\begin{proposition}
Let $\Omega\subset\CC^{2}\simeq\mathbb{H}$ be a complex non-empty Zariski open domain. Let $J:\Omega\to SO(4)/U(2)$ be a map with real algebraic entries of the variable $(q_{1},q_{2})\in\Omega$ representing an OCS, then $J$ is an algebraic OCS.
\end{proposition}

\begin{proof}
Let $q=(q_{1},q_{2})\in\Omega$. As a matrix $J$ is such that $J\in SO(4)$ and $J=-^{t}J$:
\begin{equation*}
J(q)=\left(\begin{matrix}
0 & A(q) & B(q) & C(q)\\
-A(q) & 0 & C(q) & -B(q)\\
-B(q) & -C(q) & 0 & A(q)\\
-C(q) & B(q) & A(q) & 0
\end{matrix}
\right),
\end{equation*}
where $A, B$ and $C$ are real-valued polynomials in $q$, such that $A^{2}(q)+B^{2}(q)+C^{2}(q)\equiv 1$.
Since $J$ is an OCS, then, thanks to Theorem~\ref{OCS} it corresponds to a holomorphic function $\tilde J:\Omega\to\CC\PP^{1}$ (see~\cite[Propositions 2.1 and 2.2]{sv1}).
We want now to explicit $\tilde J$ on an affine subset.

Assume that $A\not\equiv 1$, then (up to restricting to a Zariski-open subset $\Omega_{1}\subseteq \Omega$),
$$
\tilde J(q)= [1,z(q)=C(q)+iB(q)]\in\CC\PP^{1}.
$$
hence $z(q)$ is a holomorphic polynomial.

Consider now the graph $\Gamma:=\{((q_{1},q_{2}),[1,z(q)])\in\CC^{2}\times \CC\PP^{1}\,|\, (q_{1},q_{2})\in\Omega_{1}\}$. This embeds in 
$\CC\PP^{3}$ as 
$$((q_{1},q_{2}),[1,z(q)])\hookrightarrow [1,z(q),q_{1},q_{2}]$$ 
Therefore, on the affine set $X_{0}\neq 0$ it is contained in the complex algebraic surface,
defined by 
$$
z(X_{2},X_{3})-X_{1}=0.
$$
Homogenizing this equation with respect to $X_{0}$ we find a globally defined projective hypersurface containing $\Gamma$. 
\end{proof}

Starting from these introductory material, using techniques from algebraic geometry, in this paper we prove three general results on the topology of the
twistor discriminant locus of an algebraic surface $X\subset\CC\PP^{3}$. 
The following three sections are almost independent and are summarized as follows.
In Section~\ref{S1} we prove a result about the real dimension of the twistor discriminant locus of a given algebraic surface. Given an integral degree $d$ algebraic surface $X$, it is known that $\mathrm{Disc}(X)$ is always nonempty and
that $\dim(\mathrm{Disc}(X))\le 2$ (see~\cite{armstrong}). In Theorem~\ref{a1} we prove that the pure dimension of $\mathrm{Disc}(X)$ is always equal to 2 except for the cases in which $d=1$ (in which is just a point) or $d=2$
and $X$ is $j$-invariant (in which $\mathrm{Disc}(X)$ is a \textit{round} $S^{1}\subset S^{4}$).

Afterwards, in Section~\ref{S2}, we prove that given a general degree $d$ surface $X$, its possible intersections with the set of twistor lines are just of 4 kind: $d$ distinct points, a double point and $d-1$ distinct other points, two (distinct) double points and other $d-2$ distinct other points or a triple point and, again, $d-2$ distinct other points. For each of these configuration we compute the dimension in the space of twistor lines embedded in the Grassmannian $G(2,4)$ of lines in $\CC\PP^{3}$.

Section~\ref{S3} is devoted to give a decomposition of the twistor discriminant locus of a given cone. 
A cone $X$ of degree $d>2$ contains at most one twistor line. In fact two distinct twistor lines are disjoint and
all the lines contained in a cone pass through the vertex $o$. If a surface $X$ is ruled and contains infinite twistor lines 
then, either $X$ is the $j$-invariant non-singular quadric or a non-conical ruled surface, hence non-normal (see~\cite[Section 4]{ab}).
First of all we construct the dual twistor fibration, i.e.: the natural dual map $\eta$ of $\pi$, from the dual space $\CC\PP^{3\vee}$ to $S^{4}$. Then we prove that for any given degree $d$ cone $X$, if $o$ denotes its vertex, 
then $\mathrm{Disc}(X)$ is the union of $\pi (\mathrm{Sing}(X)\setminus \{o\})$ with $\eta(X^{\vee})$, $X^{\vee}$
being the dual of $X$ in $\CC\PP^{3\vee}$.


\section{Dimension}\label{S1}

In~\cite[Section 4]{armstrong} the author proves some result on the topology of the twistor discriminant locus of an algebraic surface. For instance he proves that if a degree $d$ surface $X:=\{f=0\}\subset \CC\PP^{3}$ is such that
its defining function $f$ is a square-free polynomial, then $\mathrm{Disc}(X)$ has real dimension less or equal to 2 (see~\cite[Proposition 4.2]{armstrong}). Moreover,
if $d>2$ and $X$ is non-singular, then  $\dim(\mathrm{Disc}(X))=2$ (see~\cite[Proposition 4.3]{armstrong})

In this section we improve these results showing that, with the exception of the cases of a hyperplane
or a $j$-invariant non-singular quadric, then $\dim(\mathrm{Disc}(X))$ has always pure dimension equal to 2
for any integral algebraic surface $X$.
Afterwards we also prove that, if $X$ is non-integral and $\dim(\mathrm{Disc}(X))\le 1$, then $\mathrm{Disc}(X)$
is, in fact, just a point $q$ and $X$ is a union of hyperplanes all intersecting at $\pi^{-1}(q)$.

We start by setting some notation. For any algebraic surface $X$ we set
$u:= \pi| _{X}$, $\mathrm{Disc} (X)' := u^{-1}(\mathrm{Disc} (X))$ and 
$v:= u|_{X\setminus \mathrm{Disc} (X)'}$. Thus $v: X\setminus \mathrm{Disc} (X)'\to S^4\setminus \mathrm{Disc} (X)$
is a degree $d$ unramified covering  (in the sense of the fundamental group). $\mathrm{Disc} (X)$ is a compact semialgebraic set by \cite[Theorem 2.2.1]{bcr} or \cite[Corollary 2.4 (2)]{coste}. It is in fact a real algebraic set since
it is defined by a system of two real algebraic equation (see~\cite[Section 3]{armstrong})



%

As already said in Example~\ref{hyperplane}, there is no global section for the twistor bundle of $S^{4}$ and each hyperplane $X$ contains a unique twistor line $L$. 
Therefore $\mathrm{Disc}(X)=\{\pi(L)\}$. In the following proposition we show that this remark is in fact more general, i.e. it is not allowed, for a general surface $X$ of degree $d>1$ to
have an isolated point in its twistor discriminant locus.

\begin{proposition}\label{e4}
Let $X\subset\CC\PP^{3}$ be an integral surface of degree $d$ and $q\in \mathrm{Disc} (X)$ be an isolated point of $\mathrm{Disc} (X)$. Then $d=1$, $\mathrm{Disc} (X) =\{q\}$ and $\pi ^{-1}(q)$ is the unique twistor line of the plane $X$.
\end{proposition}

\begin{proof}
First of all, let $X$ be any hyperplane. Then $\mathrm{Disc}(X)=\{q\}$ and $\pi^{-1}(q)$ is the only twistor line contained in $X$.

Assume, then, that $d>1$.  Let $U$ be an open neighborhood of $q$ in $S^{4}$ homeomorphic to an open ball with $U\cap \mathrm{Disc}(X)=\{q\}$. It is possible to find such $U$ because $q$ is isolated in $\mathrm{Disc}(X)$. Recall that $u:= \pi|_{X}$ and set $V:=u^{-1}(U)$.
Since $U\setminus \{q\}$ is homotopically equivalent to $S^3$, then it is simply connected. Hence $V\setminus u^{-1}(q)$ is a disjoint union of $d$ connected components, each of them mapped homeomorphically onto $U\setminus \{q\}$.
Assume for the moment that the twistor line $\pi^{-1}(q)$ is not contained in $X$. Then $u^{-1}(q)$ is the union of
$\tilde d <d$ points.
Since $u(\mathrm{Sing}(X))\subset \mathrm{Disc}(X)$ (see Remark~\ref{sing}), then $V\cap \mathrm{Sing}(X)$ is contained in $u^{-1}(q)$ and each point is an isolated singularity.

\quad \emph{Claim:} For any finite set $S$, the number of connected components of $V$ and $V\setminus S$ is the same.

\quad \emph{Proof of Claim:} Either $V\subset X_{\mathrm{reg}}$ or $V$ contains some isolated singular point of $X$ which are necessarily normal (see e.g.~\cite{milnor}). In both cases, the link of any point in a normal complex surface is connected and homeomorphic to a topological $3$-manifold~\cite{mumford}.


%

Since $u|_{V}$ is proper, then all connected components of $V$ are open and closed in $V$ and hence $V$ may have at most $\tilde d$ connected components. 

Assume now that the twistor line line $L:= \pi ^{-1}(q)$ is contained in $X$. We may use the proof of the Claim with  $L$  instead of $S$, because $L\cong \CC \PP^1$ and
for each normal point $p\in X$, say with link $E$, the complement in $E$ of finitely many circles (in particular $L\cap E$) is connected.
\end{proof} 

We start now to investigate the nature of the twistor discriminant locus of a degree $d>1$ integral surface. In the following lemma we prove that, in the presence of a finite number of twistor lines, the set $\mathrm{Disc}(X)$ has no ``hairs''. Before stating it, notice that if $L$ is any line in $\CC\PP^{3}$, then either $L$ is twistor and $\pi(L)$ is a
single point, or $\pi(L)$ is a round 2-sphere $S^{2}\subset S^{4}$.

\begin{lemma}\label{aa2}
Let $X\subset \CC \PP^3$ be an integral degree $d>1$ surface such that $X$ contains only finitely many twistor lines. If $\mathrm{Disc} (X) =A\cup B$ with $A, B$ closed real algebraic
subsets
and $\dim(A) \le 1$. Then $A\subseteq  B$.
\end{lemma}

%

\begin{proof}
As already said, we have that $\mathrm{Disc} (X)\ne \emptyset$ and $\dim (\mathrm{Disc} (X))\le 2$. By Proposition \ref{e4} $\mathrm{Disc} (X)$ has no isolated points. Assume that
the algebraic set $A$ has pure dimension equal to $1$ and that $A\nsubseteq B$. 
Then, there exists a closed semialgebraic set $E\subseteq X$ semialgebraically equivalent to the interval $[0,1]$ and such that $E\cap D$ is finite, where $D$ is the closure (in $A\cup B$ or in $S^4$) of $(A\cup B)\setminus E$ (\cite[Section 2.3]{bcr}).

As before, consider the restriction $u := \pi _{|X}$, and the sets $E':=u^{-1}(E)$ and $D':= u^{-1}(D)$. Notice that $u_{|X\setminus (E'\cup D')}: X\setminus (E'\cup D')\to S^4\setminus (E\cup D)$ is
an unramified degree $d$ covering (in the sense of the fundamental group). 
By assumption, for all but finitely many point of
$E\setminus E\cap D$ the associated twistor line is not contained in $X$. We claim that for all except finitely many
$q\in E'$, $q$ is a smooth point of $X$. Assume that this is not true. Since $\mathrm{Sing}(X)$ is a closed complex algebraic
subset, there would exist a complex irreducible variety $T\subseteq \mathrm{Sing}(X)$ containing a non-empty open subset of $E'$. Since $E'\cap T$
is infinite and $T$ is not a union of twistor lines then $\pi (T)$ has real dimension $2$.
Since $\pi ( \mathrm{Sing}(X)) \subseteq \mathrm{Disc} (X)$, we get that a non-empty open subset of $E$ is contained in $D$, contradicting the fact that
$E\cap D$ is finite. Thus there is $q\in E$ such that $u^{-1}(q) \subset X_{\mathrm{reg}}$ and the twistor line associated to $q$ is not contained in $X$.
Therefore the set $u^{-1}(q)$ is a non-empty finite set with cardinality $\tilde d<d$. Take a contractible closed neighborhood $U$ of $q$ such that $U\cap D=\emptyset$
and set $V:= u^{-1}(U)$, $F:= E\cap U$ and $F':= E'\cap V$. Fix $x_0\in V\setminus F'$ and set $y_0:= \pi (x_0)$. Since $F$ is topologically embedded as a closed interval in $U$,
the natural map between fundamental groups $\pi _1(U\setminus F,y_0) \to \pi _1(U,y_0)$ is bijective. Thus $V\setminus F'$ is a disjoint union of $d$ simply connected topological spaces, each of them mapped homeomorphically onto $U\setminus F$.  Since $F'$ is a finite union of closed intervals and $V \subset X_{\mathrm{reg}}$, then $V\setminus F'$ and $V$ have the same number $d$ of connected components, contradicting
the connectedness of $U$, the properness of $u$ and that $u^{-1}(q)$ has $\tilde d$ preimages.
\end{proof}

We are now in position to state and prove the main result of this section.

\begin{theorem}\label{a1}
Let $X\subset\CC\PP^{3}$ be an integral algebraic surface of degree $d$. Then, one of the following mutually exclusive statements hold:
\begin{enumerate}
\item  $\mathrm{Disc} (X)$ is finite if and only if $|\mathrm{Disc} (X) | =1 $ if and only if $ d=1$;
\item $\dim(\mathrm{Disc}(X))=1$ if and only if $X$ is conformally equivalent to the Segre quadric $\mathcal{Q}$
and in this case $\mathrm{Disc}(X)$ is conformally equivalent to $S^{1}\subset S^{4}$;
\item in all other cases $\mathrm{Disc}(X)$ is a real algebraic compact set of pure dimension 2.
\end{enumerate}

\end{theorem}

\begin{proof} Since $\mathrm{Disc}(X)$ is a non-empty compact real algebraic set and $\dim(\mathrm{Disc}(X))\leq 2$, the theorem in the case in which $X$ has only finitely many twistor lines is 
a consequence of Lemma \ref{aa2}. From now on in this
proof we assume that $X$ has infinitely may twistor lines. By the
conformal classification of quadrics (\cite{chirka,sv1}) we get the theorem if $d=2$. Thus we assume $d>2$. The proof of Lemma
\ref{aa2} shows that it is sufficient to prove the following claim.

\quad \emph{Claim:} There is no semialgebraic set $E\subset \mathrm{Disc} (X)$ semialgebraically equivalent to the closed
interval $[0,1]$, intersecting the closure of $\mathrm{Disc} (X)\setminus E$ at finitely many points and parametrizing infinitely many
twistor lines.

\quad \emph{Proof of Claim:}  Assume the existence of such $E$. Since $\pi(Sing(X))$ is contained in $\mathrm{Disc} (X)$, to get a contradiction
it is sufficient to prove that all twistor lines of $X$, except finitely many, meet an irreducible component $T$ of
$\mathrm{Sing}(X)$ with
positive dimension and which is not twistor line (so $\pi (T)$ has real dimension $2$). All but finitely many lines of $X$
(and so all but finitely many twistor lines of $X$) belong to a ruling  of $X$ and, since $d>2$,  this ruling is unique.
Call $\Ff$ the family of all lines $L\subset X$ belonging to the ruling of $X$. 
The family $\Ff$ is parametrized by an irreducible curve $D$. 
Since $\dim(\mathrm{Sing}(X)) \le 1$, only finitely many elements of $\Ff$ are singular. No two
different twistor lines meets, so, to prove the existence
of an irreducible component $T$ of $\mathrm{Sing}(X)$ which is not a twistor line, meeting all but finitely many $L\in \Ff$, it is sufficient to prove
that a general $L\in \Ff$ meets $\mathrm{Sing}(X)$. Assume that this is not true. Thus $L$ is contained in the smooth locus
$X_{\mathrm{reg}}$ of
$X$ and hence the normal bundle $N_L$ of $L$ in $X$ is a rank $1$ locally free sheaf on $L$ whose degree is the intersection
number $L^2$ (this intersection number is well-defined and it is an integer, because $L\subset X_{\mathrm{reg}}$). The line $L$
belongs to a positive-dimensional family of curves of $X$, hence we have $L^2\ge 0$. Since $X$ is a degree $d$ surface of $\CC
\PP^3$,
the dualizing sheaf $\omega_{X}$ is a well-defined line bundle on $X$ (independently of the singularities of $X$), and thanks to the adjunction's formula $\omega _X \cong
\Oo _X(d-4)$ (see~\cite[Adjunction Formula II, p 147]{GH}). We have the following equalities:
\begin{itemize}
\item Since $L\cong \CC\PP^1$, we have $\deg (\omega _L) =-2$.
\item Since $L$ is a line, we have $L\cdot \omega _X = \deg (\omega _{X|L}) =d-4$.
\end{itemize}
Finally, the adjunction formula gives $L^2+L\cdot \omega _X = -2$, but $L^2\ge 0$, therefore we get $d\le 2$, a contradiction.
\end{proof}

\begin{remark}
If $X$ is an integral normal surface, then Theorem~\ref{a1} is just a corollary of Proposition~\ref{e4} and Lemma~\ref{aa2}.
\end{remark}

\subsection{Non-Integral Case}
In this part of the section we study the case of non-integral surfaces. We have the following trivial remark.
\begin{remark}
If the surface $X$ has a multiple component, clearly $\mathrm{Disc}(X)=S^{4}$.
\end{remark}
Being, na\"ively speaking, the union of two
or more (integral) surfaces, it is clear that the twistor discriminant locus of a non-integral surface is much rigid.

The following two results show this behavior and allow us to precise the cases in which $\mathrm{Disc}(X)$ is a single point
or has dimension less or equal to 1. Thanks to these two results, at the end of this section, we are able to give a 
(conformal) characterization of algebraic OCS's defined outside a 1-dimensional set.

\begin{example}
Let $H_{1}$ and $H_{2}$ be two planes containing the twistor lines $L_{1}$ and $L_{2}$, respectively.
As in~\cite{chirka}, if $L_{1}\neq L_{2}$, then $\mathrm{Disc}(H_{1}\cup H_{2})=\pi(H_{1}\cap H_{2})$. Moreover 
since $H_{1}\cap H_{2}=L$ is a projective line and any two lines in a plane intersect, then both $L\cap L_{1}$
and $L\cap L_{2}$ are nonempty. Eventually, since $L_{n}$ is the unique twistor line in $H_{n}$, $n=1,2$, then
$\mathrm{Disc}(H_{1}\cup H_{2})$ is a round 2-sphere $S^{2}$ with two marked points $p_{1}=\pi(L_{1}), p_{2}=\pi(L_{2})$.
Any other point $p\in S^{2}\setminus\{p_{1},p_{2}\}$ is such that $|\pi^{-1}(p)|=2$.
\end{example}

For non-integral surfaces we have the following result which can be considered a reciprocal of~\cite[Theorem 1.3]{sv1}
from the point of view of the twistor bundle.
\begin{proposition}\label{e1}
Let $X\subset \CC\PP^3$ be any degree $d$ surface. $\mathrm{Disc}(X)$ is a finite set if and only if either $d=1$ or $d>1$ and $X$ is the
union of $d$ distinct planes $H_1,\dots ,H_d$ containing the same twistor line $L$. In all cases $\mathrm{Disc}(X)$ is a single point.
\end{proposition}

\begin{proof}
The case $d=1$ and the ``if'' part are obvious. Let now $d$ be greater than one and $\mathrm{Disc}(X)$ be a finite set. Assume that
Proposition \ref{e1} is true for all surfaces of degree $<d$. Let $\mathrm{Disc}(X)'=\pi^{-1}(\mathrm{Disc}(X))$. Since, for any finite set $A$,
$S^4\setminus A$ is simply connected and 
$\pi|_{X\setminus \mathrm{Disc}(X)'}:X\setminus \mathrm{Disc}(X)'\to S^{4}\setminus \mathrm{Disc}(X)$ is an unramified covering of degree
$d$, $X\setminus \pi ^{-1}(\mathrm{Disc}(X))$ must have $d$ connected component. For any $q\in X$ there is a fundamental system $\Uu$ of open
neighborhoods
of $q$ in $X$ for the euclidean topology such that $U\setminus \{q\}$ is connected for all $U\in \Uu$ (use a triangulation of
the complex algebraic subset $X$). Thus $\mathrm{Disc} (X)'$ is a union of $|\mathrm{Disc}(X)|$ twistor lines. If $X$ is a union of
planes, then we immediately get that $X$ is as in the configuration described in Proposition \ref{e1}. Now assume that $X$ has
an irreducible component $Y$ of degree $>1$. By the inductive assumption we see that $Y=X$. To get a contradiction it is
sufficient to prove that $Y\setminus E$ is connected for any union $E$ of finitely many lines. Take a general plane $H\subset
\PP^3$. In particular
$Y\cap E$ is a finite set. Since $H$ is general, Bertini's theorem implies that $X\cap H$ is an integral complex curve.
Now $E\cap H$ is finite and $X\cap H$ is an integral curve, therefore $(X\cap H)\setminus (E\cap H)$ is connected by arc.
Varying $H$
we get that $X\setminus E$ is connected by arcs.
\end{proof}

\begin{proposition}\label{e3}
Let $X$ be a reducible degree $d$ surface such that $\dim(\mathrm{Disc} (X))\le 1$. Then $X$ is as in Proposition \ref{e1} and hence $|\mathrm{Disc} (X)|=1$.
\end{proposition}

\begin{proof}
Since $\mathrm{Disc} (X)\ne S^4$, $X$ has no multiple component. Call $X_1,\dots ,X_h$, $h\ge 2$, the irreducible components of $X$. Since $ \mathrm{Disc}(X_i)\subseteq \mathrm{Disc} (X)$,
each $X_i$ is either as in Proposition \ref{e1} or as in Theorem~\ref{a1} case (2), i.e. $\mathrm{Disc} (X_i) =S^1$ and $X_i$ is conformally equivalent to the Segre quadric $\mathcal{Q}$. First assume that one of the components, say $X_1$, is equivalent to $\Qq$. Call $|\Oo _{X_1}(1,0)|$ the pencil of lines of $X_1$ containing the twistor lines \cite[Section 7.3]{sv1}. The component
$X_2$ is either as in Proposition \ref{e1} or, again, equivalent to $\Qq$. Set $e:= \deg (X_2)\in \{1,2\}$. The scheme $X_1\cap X_2$ is an element of $|\Oo _{X_1}(e,e)|$ and hence
it cannot be a union of twistor lines with certain multiplicities. Thus $\mathrm{Sing}(X)$ contains a complex curve which is not a union of twistor lines. Thus $\dim(\mathrm{Disc} (X))\ge 2$,
a contradiction. Now assume that each $X_i$ is a plane. As above we see that each $X_i\cap X_j$, $1\le i<j\le h$, is a twistor line. Since each plane contains a unique twistor line, $X$ is as in Proposition \ref{e1}.
\end{proof}

Thanks to these two results it is possible to state the following corollary which show a peculiar behavior of algebraic OCS's.

\begin{corollary}
Up to conformal transformations, the circle is the only possible 1-dimensional twistor discriminant locus of an algebraic OCS on $S^{4}$.
\end{corollary}

%

\section{Stability of the Discriminant Locus}\label{S2}
In this section we analyze the possible intersections between a general surface $X$ and the set of twistor lines.
We will show that only few configurations are allowed. For any of these configurations it is possible to compute their 
dimension. 
Let us denote by $G(2,4)$ the Grassmannian of lines in $\CC\PP^{3}$ and by $\sigma:S^{4}\to G(2,4)$ the $\mathcal{C}^{\infty}$ embedding
of twistor lines, i.e. $\sigma(q):=\pi^{-1}(q)$ (see~\cite{altavillatwistor,gensalsto} for an explicit definition of $\sigma$).  We recall that $\dim(G(2,4))=4$. Note that $\sigma(S^4)$ is a compact differential submanifold of $G(2,4)$ with real codimension equal to $4$. 
The action of $SL(4,\CC)$ on $\CC\PP^3$ is transitive on the set of lines, hence it induces a natural action on $G(2,4)$.

%
%


Our main result is based on the following two remarks.

\begin{remark}\label{a00} Since the action of $SL(4,\CC)$ on $G(2,4)$ is transitive, then, for each $L\in G(2,4)$, the map $SL(4,\CC)\to G(2,4)$ given by $g\mapsto g(L)$ is a submersion.
Thus the map $\phi : SL(4,\CC)\times S^4\to G(2,4)$ defined by $\phi (g,q) = g(\sigma(q))$ is a submersion too. Hence for each locally closed connected complex differential submanifold $T\subset G(2,4)$, $\phi ^{-1}(T)$
is smooth and it has codimension $4-\dim (T)$ (\cite[Theorem 2.8]{gg}). Hence, given a general $g\in SL(4,\CC)$,
 $g(T)$ intersects transversally $\sigma(S^4)$. Thus if $\dim_{\CC} (T)\le 1$, then $g(T)\cap S^4=\emptyset$ for a general $g\in SL(4,\CC)$, while
if $\dim (T)\ge 2$, then $g(T)\cap \sigma(S^4)$ has pure dimension $\dim (T) -2$
(it could be disconnected, a priori).
\end{remark}

We now pass to describe the possible intersections between a general surface and a line.
Fix an integer $d\ge 2$ and let $\alpha=(m_{1},\dots,m_{h})$, with $m_1\ge \cdots \ge m_h$, be a partition of $d$, i.e. $h$ is a positive integer, the sequence $m_i$ is non-decreasing, $m_h>0$ and
$m_1+\cdots +m_h =d$. For any line $L\subset \CC\PP^3$ and any zero-dimensional scheme $Z\subset L$ let $\alpha (Z)$ be the partition of $d$ obtained in the following
way. Set $S:= Z_{\mathrm{red}}$ and $h:= |S|$. For each $q\in S$ let $m_q$ be the multiplicity of $Z$ at $q$, i.e. write $Z = \sum _{q\in S} m_qq$ as effective divisor of $L\cong \CC\PP^1$.
We order the point of $S$, say $S = \{q_1,\dots ,q_h\}$ so that $m_i\ge m_j$ for all $i\le j$. Define  $\alpha (Z) = (m_{1},\dots,m_{h})$.

\begin{remark}\label{a01}
For any partition $\alpha$ of $d$ let $L(\alpha )$ be the set of all zero-dimensional schemes $Z\subset L$ such that $\deg (Z)=d$, $\alpha (Z) =(m_{1},\dots,m_{h})$ and $m_{1}+\dots m_{h}=d$. Since the set of all $h$-ples
of distinct points of $L$ is a smooth complex manifold with complex dimension $h$, $L(\alpha )$ is a complex manifold with complex dimension $h$. Note that the set $S^dL$
of all degree $d$ effective divisors of $L$ is a smooth compact manifold of complex dimension $d$. We see $L(\alpha )$ as a locally closed complex submanifold of $S^dL$
with pure complex codimension $d-h$.
\end{remark}

%

We now prove the main result of this section.

\begin{theorem}\label{a2}
Let $X\subset \CC\PP^3$ be a general degree $d\ge 4$ surface. Then $X$ contains no line and for any $L\in G(2,4)$ the partition $\alpha (L\cap X)$ is either $(1,\dots ,1)$
occurring in a Zariski open subset of $G(2,4)$, or one of the following:
\begin{itemize}
\item $(2,1,\dots ,1)$ occurring on a locally closed hypersurface $T_{1}^{1}\subset G(2,4)$;
\item $(2,2,1,\dots ,1)$ occurring in a $2$ dimensional locally
closed complex subvariety $T_{2}^{1}\subset G(2,4)$;
 \item $(3,1,\dots ,1)$ occurring in a $2$ dimensional locally
closed complex subvariety $T_{2}^{2}\subset G(2,4)$;
\item $(3,2,1,\dots ,1)$ occurring in a $1$-dimensional locally
closed complex subvariety $T_{3}^{1}\subset G(2,4)$;
\item $(4,1,\dots ,1)$ occurring in a $1$-dimensional locally
closed complex subvariety $T_{3}^{2}\subset G(2,4)$;
\item $(2,2,2,2,1,\dots ,1)$ occurring in a finite set $T_{4}^{1}\subset G(2,4)$;
\item $(3,2,2,1,\dots ,1)$ occurring in a finite set $T_{4}^{2}\subset G(2,4)$;
\item  $(4,2,1,\dots ,1)$ occurring in a finite set $T_{4}^{2}\subset G(2,4)$;
\item $(5,1,\dots ,1)$ occurring in a finite set $T_{4}^{2}\subset G(2,4)$.
\end{itemize}
\end{theorem}

\begin{proof}
For any $d\ge 4$ it is easy to check that the set of all degree $d$ surfaces containing at least one line has codimension $d-3$ in the projective space of all degree $d$ surfaces, even the singular or reducible ones (see \cite{xu} for a deeper result; for very general surfaces it is also a consequence of Noether-Lefschetz theorem). In particular a general surface of degree $>3$ contains no line. 

Let us now consider, in the projective space of all degree $d$ surfaces in $\CC\PP^{3}$, the non-empty Zariski open 
set $\Uu_{d}$ of smooth surfaces containing no line. 
Let now $L$ be any line and $\alpha$ be any partition of $d$.
In $\Uu_{d}$ we identify with $\Sigma_{L, d, \alpha}$, the family of surfaces $S$ such that $S\cap L=\alpha$. 
We know that $\Sigma_{L, d, \alpha}$ has codimension $d-h$ in $\Uu_{d}$. 
Consider now the set $\Sigma _{d,\alpha} \subset G (2,4)\times\Uu _d$ composed by pairs $(L,S)$ such that $S\cap L=\alpha$. Clearly, the codimension of $\Sigma _{d,\alpha}$ in $G (2,4)\times\Uu _d$ is again $d-h$. 
Therefore the projection in the second factor $p_{2}(\Sigma _{d,\alpha})$ has codimension, at least, $d-h-4$.
Hence, if $d-h>4$, then $\Uu_{d}\setminus p_{2}(\Sigma _{d,\alpha})$ contains a non-empty Zarisky open subset
$V_{d,\alpha}$, such that no $S\in V_{d,\alpha}$ has partition $\alpha$ with respect to any line.

Therefore, we have that $d-4\le h\le d$ and so all the possible partitions are those listed in the statement.
%
%
%
%
%
%
%
%
%

\end{proof}

Notice that we do not claim non-emptiness for the partitions occurring in Proposition \ref{a2}. Of course, when $d=4$ the partition $(5,1,\dots ,1)$ never occurs as well as when $d=3$ no partition $(2,2,1,\dots ,1)$ or higher occur.

We pass now to the twistor case. It is clear that, since the set of twistor lines $\sigma(S^{4})$ is a submanifold of $G(2,4)$, then the previous result will return a much more refined list. 

\begin{corollary}\label{a3}
Let $X$ be a general degree $d\ge 4$ surface. Then $X$ contains no twistor line and for every twistor line $L$ either $\alpha (L\cap X)=(1,\dots ,1)$ occurring in a Zariski open subset of $\sigma(S^{4})\subset G(2,4)$, or one of the following:
\begin{itemize}
\item $\alpha (L\cap X)=(2,1,\dots ,1)$ occurring on a $2$-dimensional linear subvariety of $\sigma(S^{4})$;
\item $\alpha (L\cap X) =(2,2,1,\dots ,1)$ occurring only for finitely many twistor lines;
\item $\alpha (L\cap X)=(3,1,\dots ,1)$ occurring only for finitely many twistor lines.
\end{itemize}
\end{corollary}

\begin{proof}
We use Remarks \ref{a00} and \ref{a01} and Theorem \ref{a2} with this additional observation. The complex varieties $T_{m}^{n}$, $m,n=1,\dots,4$, occurring in the statement of Proposition \ref{a2} are not necessarily
smooth. Any such variety has a stratification in unions of smooth varieties (with a Zariski open subset consisting of the smooth points of $T_{m}^{n}$, all the other strata having lower dimension).
\end{proof}

\begin{remark}
Let $X$ be degree $d$ surface defined by a square free polynomial $\{f=0\}$.
Our result is, then, compatible with~\cite[Proposition 4.2]{armstrong} in which it is stated that the set of double points $\Dd':=\{q\in \mathrm{Disc}(X)\,|\, |\pi^{-1}\cap X|=d-1\}$ is a smooth orientable real surface.
\end{remark}

\section{Twistor Discriminant Locus of Algebraic Cones}\label{S3}

Let $\mathbb {CP}^{3\vee}$ denote the dual projective space of $\mathbb {CP}^{3}$, i.e. the $3$-dimensional complex projective
space
parametrizing the set of all planes of $\mathbb {CP}^{3}$. 
For any integral closed complex algebraic variety $Y\subset \mathbb {CP}^{3}$, let $Y^\vee \subset \mathbb {CP}^{3\vee}$
denote its dual variety. 
 For definition and simple properties of the dual of a plane curve, see  \cite[page 264]{GH} and \cite[Ex. I.7.3]{h}. For dual of hypersurfaces (resp. arbitrary varieties) in any projective space see \cite[\S 1.2 and 1.3]{dolg}.
 
In this section we give a first account of the twistor discriminant locus of a projective cone $X$ in terms of its dual variety $X^{\vee}$ and of its singular locus. 

Results about ruled surfaces from this point of view were given in~\cite[Section 4]{ab}. For the particular case
of cones, the only known results in our knowledge are~\cite[Theorem 31]{altavillatwistor} about cubic cones and 
their possible parameterization by slice regular functions and
 the conformal classification of quadric cones~\cite{chirka} which is summarized in the following example.
 
 \begin{example}\label{examplecone}
Let $X$ be a quadric cone. We have the following two cases:
\begin{itemize}
\item if $X$ contains a twistor line, then $\mathrm{Disc}(X)$ is a round 2-sphere in $\mathbb{S}^{4}$ and all such cones are conformally equivalent (see~\cite[Statement 7.3.1]{chirka});
\item If $X$ does not contain any twistor line, then $\mathrm{Disc}(X)$ is homeomorphic to a $2$-sphere with two identified points. In this case there is a 1-dimensional continuous family of conformally inequivalent cones (see~\cite[Statement 7.3.2]{chirka}).
\end{itemize}
 \end{example}

In view of the main result of this section, we need to recover the standard twistor fibration in terms of $\CC\PP^{3\vee}$.
We recall the definition of the \textit{inclusion variety} (or \textit{incidence correspondence}),
as the set $A$ defined by
$$
A:=\{(L,H)\in G(2,4)\times \CC\PP^{3\vee}\,|\, L\subset H\}.
$$
In this set we identify the family of twistor lines as 
$$
B:=\{(L,H)\in A\,|\, L\in\sigma(S^{4})\}.
$$
The projections onto the two factors $A\to G(2,4)$ and $A\to \CC\PP^{3\vee}$ induce complex algebraic maps
$$
\eta_{1}:B\to\sigma(S^{4}),\qquad \eta_{2}:B\to\CC\PP^{3\vee},
$$
which are locally trivial fibrations (in the
Zariski topology and hence in the euclidean topology too), with fiber biholomorphic to $\mathbb {CP}^{1}$
and $\mathbb {CP}^{2}$, respectively. Since each plane $H$ contains a unique twistor line $L$, the map
$\eta_{2}$ is proper and bijective and hence a homeomorphism. Moreover, since $\sigma$ is real algebraic,
then $\eta_{2}$ is also real algebraic and, being bijective with a smooth domain and a smooth target, it is a diffeomorphism.

The dual twistor fibration is then given by the map $\eta:\CC\PP^{3\vee}\to S^{4}$, defined by
$$
\eta:=\sigma^{-1}\circ\eta_{1}\circ \eta_{2}^{-1}.
$$
Since, for any line $L$, the set of planes $H$ such that $L\subset H$ is biholomorphic to $\CC\PP^{1}$, then the 
fibers of $\eta$ are biholomorphic to $\CC\PP^{1}$. Notice that, for each plane $H\subset \CC\PP^{3}$, the set of 
all lines contained in $H$ is biholomorphic to $\CC\PP^{2}$ and this is, in fact, $H^{\vee}$.

The map $\eta$ is then a locally trivial
fibration (in the euclidean topology) real algebraic with fibers homeomorphic to $\mathbb {CP}^{1} \cong S^2$. Thus for each
closed set $\Sigma \subseteq \mathbb {CP}^{3\vee}$ (in the Zariski or the euclidean topology), we get a compact subset $\eta (\Sigma)\subseteq S^4$.
If $\eta _{|\Sigma}$ is injective, then $\eta _{|\Sigma}: \Sigma \to \eta (\Sigma )$ is a homeomorphism, because it is continuous and closed (since $\Sigma$ is compact
and $S^4$ is Haussdorff).

We now set the notation used in the reminder of the section. 
Let $X$ be an integral cone surface of degree $d>1$ and call $o\in\CC\PP^{3}$ its vertex. 
Take any plane $H\subset \CC\PP^{3}$ such that $o\notin H$ and denote by $M_{o}\subset\CC\PP^{3\vee}$ the hyperplane of $\CC\PP^{3\vee}$
formed by all planes of $\CC\PP^{3}$ containing $o$. 
Consider the integral curve $D:= H\cap X$. We have that $deg (D)= deg (X)$. 

Notice that every line contained in $X$
pass through $o$ and since two different twistor lines are disjoint, then $X$ contains at most one twistor line $\mathcal L\ni o$.

For any $a\in X\setminus \{o\}$ let $L_a$ be the line spanned by $\{a,o\}$. Since $X$ is a cone with vertex $o$, the point $a$ is a smooth point of $X$ if and only if $L_a\cap H$ is a smooth point of $D$. Thus if $a$ is a smooth point of $X$, then $X$ is smooth at each point $b\in L_a\setminus \{o\}$, the tangent planes $T_aX$ and $T_bX$ are the same and they are the planes spanned by $o$ and the tangent line of $D$ at the point $L_a\cap H$.

Thus the dual variety $X^{\vee}\subset\CC\PP^{3\vee}$ can be identified with $D^{\vee}\subset M_{o}$, up to the identification of $H^{\vee}$
with $M_{o}$, i.e. $X^\vee$ is the cone with vertex $o$ and $D^{\vee}\subset H^\vee$ as its basis.

We are now in position to state our main result.

\begin{theorem}\label{bbo2}

Let $X\subset \mathbb {CP}^{3}$ be an integral cone surface of degree $d>1$. Call $o\in \mathbb {CP}^{3}$ the vertex of $X$. We have the following equality
$$
\mathrm{Disc}(X)=\eta (X^{\vee})\cup \pi(\mathrm{Sing} (X)\setminus\{o\}).
$$
\end{theorem}

\begin{proof}
A line $R\subset \mathbb {CP}^{3}$ is tangent to $X$
at a smooth point $a$ of $X$ if and only if it is contained in $T_aX$. 
Take any line $R\subset T_aX$. Since $T_aX$ is a plane, we have $R\cap L_a\ne \emptyset$. Assume for the moment $o\notin R$. We get that
$R\cap L_a$ is some $b\in L_a\setminus \{o\}$ and so it is a smooth point of $X$. Since $T_aX =T_bX$, $R$ is 
tangent to $X$ at $b$. Remember that every plane contains a unique twistor line. We get that a twistor line is tangent to 
$X$ at some smooth point if and only in it is contained in some tangent plane $T_aX$, $a\in D_{\mathrm{reg}}$ and 
$R\ne L_a$. Since by the definition $\mathrm{Disc} (X)$ is closed and each plane contains a unique twistor line, we get 
that $\eta (X^\vee)\subseteq \mathrm{Disc} (X)$. If now $o\in R$ and $R$ is a twistor line, then clearly $R\subset T_{a}X$.
Since by Lemma~\ref{aa2} $\mathrm{Disc} (X)$ has pure real 
dimension $2$ and any twistor line through a singular point of $X$ is contained in $\mathrm{Disc} (X)$, we get that 
$\mathrm{Disc} (X)= \eta (X^\vee )\cup \pi (\mathrm{Sing}(X)\setminus \{o\})$.\end{proof}

Notice that if $o$ is the unique singular point of $X$ we get $\mathrm{Disc} (X) = \eta (X^\vee )$ i.e.: the projection of
the entire dual variety. 

Given any integral plane curve $D$, we have $\deg (D^\vee )\le d(d-1)$ with equality if $D$ is smooth (see \cite[Ex. IV.2.3 or V.1.5.1]{h} for the smooth case). A precise formula for $\deg (D^\vee )$ is given in \cite[p. 280]{GH} under strong assumption on the singularities and inflectional points of $D$. 

\begin{remark}
As in the proof of Theorem~\ref{bbo2}, let $D$ be the integral plane curve defined by $D=H\cap X$ and $\mathcal{L}$ be the unique twistor line through $o$. Set $\{q\}:= H\cap \mathcal{L}$. If  $L\not\subset X$, then  $q\notin D$ and there are at most $d(d-1)$
tangent lines of $D_{\mathrm{reg}}$ passing through $q$. Hence the fiber of $\eta|_{X^{\vee}}$ over $\pi(\mathcal{L})\in S^4$ in the statement of Theorem \ref{bbo2} has at most cardinality $d(d-1)$.
Assuming that $d=2$ we have that $D$ is a smooth plane conic. To recover the result in Example~\ref{examplecone} it is sufficient to observe that for every $p\notin D$ there are exactly $2$ lines of $H$ passing through
$p$ and tangent to the smooth conic $D$. In other terms, for any quadric cone that does not contain any twistor line,
there are exactly two different tangent planes intersecting at $\mathcal{L}$.

Therefore, if $o$ is the unique singular point of $X$ and $\mathcal{L}\not\subset X$, then $\eta|_{X^{\vee}}$ is a homeomorphism
outside a finite set, namely, outside $\eta^{-1}(\pi(o))$. Moreover, if the twistor line $\mathcal{L}\ni o$ is contained in $X$, then (following the proof of Theorem~\ref{bbo2}),
$\eta|_{X^{\vee}}$ is injective and, therefore, a homeomorphism onto its image.
\end{remark}

\section{Concluding Remark}

In this short paper, together with~\cite{ab}, we showed how classical methods from algebraic geometry may be
adapted in the study of (algebraic) twistor spaces. The results of Section~\ref{S1} give strong necessary conditions
on the branching locus of a surface projecting on $S^{4}$. Section~\ref{S2} tells us what to expect in the analysis
of the twistor discriminant locus of a given surface and the last Section~\ref{S3} set a possible groundwork in the
conformal classification of algebraic cones. In the future we plan to exploit these kind of techniques to other ``algebraic''
situations as the cases of $\CC\PP^{2}$ or higher dimensional spheres.

\end{document}